\documentclass[a4paper,draft,reqno,12pt]{amsart}
\usepackage[english]{babel}
\usepackage{amsmath}
\usepackage{amssymb}
\usepackage{amscd}
\usepackage{amsthm}
\usepackage{euscript}
\newtheorem{proposition}{Proposition}
\newtheorem{lemma}{Lemma}
\newtheorem{theorem}{Theorem}

\newtheorem{corollary}{Corollary}
\theoremstyle{definition}
\newtheorem{example}{Example}
\theoremstyle{remark}
\newtheorem {remark}{Remark}

\DeclareMathOperator{\dv}{div}
\DeclareMathOperator{\wdiv}{WDiv}
\DeclareMathOperator{\rk}{rk}
\DeclareMathOperator{\dg}{deg}
\DeclareMathOperator{\spec}{Spec}

\DeclareMathOperator{\aut}{Aut}

\DeclareMathOperator{\cl}{Cl}

\DeclareMathOperator{\supp}{Supp}

\def\Ker{{\rm Ker}}

\def\GG{{\mathbb G}}
\def\FF{{\mathbb F}}

\def\KK{{\mathbb K}}
\def\TT{{\mathbb T}}
\def\ZZ{{\mathbb Z}}

\def\NN{{\mathbb N}}
\def\QQ{{\mathbb Q}}
\def\PP{{\mathbb P}}
\def\AA{{\mathbb A}}

\def\OOO{\mathcal{O}}

\begin{document}
\date{}
\title[Equivariant embeddings of commutative algebraic groups]{Equivariant embeddings of commutative linear algebraic groups of corank one}
\author{Ivan Arzhantsev}
\address{National Research University Higher School of Economics, Faculty of Computer Science, Kochnovskiy Proezd 3, Moscow, 125319 Russia}
\email{arjantse@mccme.ru}
\author{Polina Kotenkova}
\address{Yandex, ulica L'va Tolstogo 16, Moscow, 119034 Russia}
\email{kotpy@mail.ru}



\maketitle

\begin{abstract}
Let $\KK$ be an algebraically closed field of characteristic zero, $\GG_m=(\KK\setminus\{0\},\times)$ be its multiplicative group, and $\GG_a=(\KK,+)$ be its
additive group. Consider a commutative linear algebraic group $\GG=(\GG_m)^r\times\GG_a$.
We study equivariant $\GG$-embeddings, i.e. normal $\GG$-varieties $X$ containing $\GG$ as an open orbit. We prove that $X$ is a toric variety and all such actions of $\GG$ on $X$ correspond to Demazure roots of the fan of $X$. In these terms, the orbit structure of a $\GG$-variety $X$ is described.
\end{abstract}

\section{Introduction}
Let $\KK$ be an algebraically closed field of characteristic zero, $\GG_m=(\KK\setminus\{0\},\times)$ be its multiplicative group, and $\GG_a=(\KK,+)$ be its additive group. It is well known that any connected commutative linear algebraic group $\GG$ over $\KK$ is isomorphic to $(\GG_m)^{r}\times(\GG_a)^s$ with some non-negative integers $r$ and $s$, see \cite[Theorem~15.5]{Hum}. We say that $r$ is the {\itshape rank} of the group $\GG$ and $s$ is the {\itshape corank} of $\GG$.

The aim of this paper is to study equivariant embeddings of commutative linear algebraic groups. Let us recall that an {\itshape equivariant embedding} of an algebraic group $G$ is a pair $(X,x)$, where $X$ is an algebraic variety equipped with a regular action $G\times X\to X$ and $x\in X$ is a point with the trivial stabilizer such that the orbit $Gx$ is open and dense in~$X$. We assume that the variety $X$ is normal. If $X$ is supposed to be complete, we speak about {\itshape equivariant compactifications} of $G$. For the study of compactifications of reductive groups,
see e.g.~\cite{Ti}. More generally, equivariant embeddings of homogeneous spaces of reductive groups is a popular object starting from early 1970th. Recent survey of results in this field may be found in~\cite{Tibook}.

Let us return to the case $\GG=(\GG_m)^{r}\times(\GG_a)^s$. If $s=0$ then $\GG$ is a torus and we come to the famous theory of toric varieties, see \cite{De}, \cite{Oda}, \cite{Fu}, \cite{CLS}. Another extreme $r=0$ corresponds to embeddings of a commutative unipotent (=vector) group. This case is also studied actively during last decades, see \cite{HT}, \cite{AS}, \cite{Arz2}, \cite{Fe}, \cite{DL}. The next natural step is to study the mixed case $r>0$ and $s>0$ and to combine advantages of both torus and additive group actions.

The present paper deals with the case $s=1$, i.e. from now on $\GG$ is a connected commutative linear algebraic group of corank one. In other words, $\GG=(\GG_m)^{n-1}\times\GG_a$, where
$n=\dim X$.

Let $X$ be a toric variety with the acting torus $\TT$. Consider an action $\GG_a\times X\to X$
normalized by $\TT$. Then $\TT$ acts on $\GG_a$ by conjugation with some character $e$. Such a character is called a {\itshape Demazure root} of $X$. If $T=\Ker(e)$, then the group $\GG:=T\times\GG_a$ acts on $X$ with an open orbit, and $X$ is a $\GG$-embedding, see Proposition~\ref{constr}. Our main result (Theorem~\ref{tmain}) states that all $\GG$-embeddings can be realized this way. To this end we prove that for any $\GG$-embedding $X$ the $(\GG_m)^{n-1}$-action on $X$ can be extended to an action of a bigger torus $\TT$ which normalizes the $\GG_a$-action and $X$ is toric with respect to $\TT$.

This result can not be generalized to groups of corank two; examples of non-toric surfaces which are equivariant compactifications of $\GG_a^2$ can be found in~\cite{DL}.
Similar examples are constructed in~\cite{DL}, \cite{DL2} for semidirect products
$\GG_m\rightthreetimes\GG_a$. Such groups can be considered as non-commutative groups of corank one.

If two toric varieties are isomorphic as abstract varieties, then they are isomorphic as toric varieties \cite[Theorem~4.1]{Be}. This shows that the structure of a torus embedding on a toric variety is unique up to isomorphism. A structure of a $\GG$-embedding on a given variety may be non-unique, see Examples~\ref{ex2}, \ref{ex4}. Such structures are given by Demazure roots and thus the number of structures is finite if $X$ is complete, and it is at most countable for arbitrary $X$. At the same time, $\GG_a^6$-embeddings into $\PP^6$ admit a non-trivial moduli space~\cite[Example~3.6]{HT}.

The paper is organized as follows. Section~\ref{sec2} contains preliminaries on torus actions on affine varieties. We recall basic facts on affine toric varieties and introduce a description of affine $T$-varieties in terms of proper polyhedral divisors due to Altmann and Hausen~\cite{AH}.
A correspondence between $\GG_a$-actions on $X$ normalized by $T$ and homogeneous locally nilpotent derivations (LNDs) of the algebra $\KK[X]$ is explained. We define Demazure roots of a cone and use them to describe homogeneous LNDs on $\KK[X]$, where $X$ is toric. Also we give a description of homogeneous LNDs of horizontal type on algebras with grading of complexity one obtained by Liendo~\cite{L1}.

In Section~\ref{sec3} we show that if $X$ is a normal affine $T$-variety of complexity one and the algebra $\KK[X]$ admits a homogeneous LND of degree zero, then $X$ is toric with an acting torus $\TT$, $T$ is a subtorus of $\TT$, and $\TT$ normalizes the corresponding $\GG_a$-action. This gives the result for affine $\GG$-embeddings. Moreover, Proposition~\ref{propaff} provides an explicit description of affine $\GG$-embeddings.

Section~\ref{sec4} deals with compactifications of $\GG$. Here we use the Cox construction and a lifting of the action of $\GG$ to the total coordinate space $\overline{X}$ of $X$ to deduce the result from the affine case.

In Section~\ref{sec5} we recall basic facts on toric varieties and introduce the notion of a Demazure root of a fan following Demazure~\cite{De}. The action of the corresponding one-parameter subgroup on the toric variety is also described there.

Let $\Sigma$ be a fan and $e$ be a Demazure root of $\Sigma$. In Section~\ref{sec6} we define a $\GG$-embedding associated to the pair $(\Sigma, e)$ and study the $\GG$-orbit structure of $X$.
It turns out that the number of $\GG$-orbits on $X$ is finite.

Finally, in Section~\ref{sec7} we prove that any $\GG$-embedding is associated with some pair $(\Sigma, e)$. The idea is to reduce the general case to the complete one via equivariant compactification. At the end several explicit examples of $\GG$-embeddings are given.

Some results of this paper appeared in preprint~\cite{AK}. They form a part of the Ph.D. thesis of the second author~\cite{Ko}.

\section{$\GG_a$-actions on affine $T$-varieties} \label{sec2}
Let $X$ be an irreducible affine variety with an effective action of an
algebraic torus $T$, $M$ be the character lattice of $T$,
$N$ be the lattice of one-parameter subgroups of $T$, and
$A=\KK[X]$ be the algebra of regular functions on $X$. It
is well known that there is a bijective correspondence between
effective $T$-actions on $X$ and effective $M$-gradings on~$A$. In
fact, the algebra $A$ is graded by a semigroup of lattice points
in some convex polyhedral cone $\omega\subseteq
M_{\QQ}=M\otimes_{\ZZ}\QQ$. So we have
$$
A=\bigoplus_{m\in \omega_{ M}} A_m\chi^m,
$$
where $\omega_{M}=\omega\cap M$ and $\chi^m$ is the character corresponding to $m$.

A derivation $\partial$ on an algebra $A$ is said to be {\itshape locally nilpotent} (LND) if for each $f\in A$ there exists $n\in\NN$ such that $\partial^n(f)=0$. For any LND $\partial$ on $A$ the map ${\varphi_{\partial}:\GG_a\times A\rightarrow A}$, ${\varphi_{\partial}(s,f)=\exp(s\partial)(f)}$, defines a
structure of a rational $\GG_a$-algebra on $A$. In fact, any regular $\GG_a$-action
on $X=\spec A$ arises this way.  A derivation $\partial$ on $A$
is said to be {\itshape homogeneous} if it respects the
$M$-grading. If ${f,h\in A\backslash \ker\partial}$ are homogeneous, then
${\partial(fh)=f\partial(h)+\partial(f)h}$ is homogeneous too and
${\dg\partial(f)-\dg f=\dg\partial(h)-\dg h}$. So any homogeneous
derivation $\partial$ has a well defined {\itshape degree} given
as $\dg\partial=\dg\partial(f)-\dg f$ for any homogeneous $f\in A\backslash \ker
\partial$. It is easy to see that an LND on $A$ is homogeneous if and only if the corresponding $\GG_a$-action is normalized by the torus~$T$ in the automorphism group~$\aut (X)$.

Any derivation on $\KK[X]$ extends  to a derivation on the field of fractions $\KK(X)$ by the Leibniz rule. A homogeneous LND $\partial$ on $\KK[X]$ is said to be of
{\itshape fiber type} if $\partial(\KK(X)^{T})=0$ and of
{\itshape horizontal type} otherwise. In other words, $\partial$
is of fiber type if and only if the general orbits of
corresponding $\GG_a$-action on $X$ are contained in the
closures of $T$-orbits.

Let $X$ be an affine toric variety, i.~e. a normal affine variety
with a generically transitive action of a torus $T$. In this case
$$
A=\bigoplus_{m\in \omega_{M}}\KK\chi^m=\KK[\omega_{M}]
$$
is the semigroup algebra.  Recall that for given cone $\omega\subset M_{\QQ}$, its {\itshape dual cone} is defined by
$$
\sigma=\{n\in N_{\mathbb{Q}}\,|\,\langle n,p\rangle\geqslant0\,\,\,\forall p\in\omega\},
$$
where $\langle,\rangle$ is the pairing between dual lattices $N$ and $M$. Let $\sigma(1)$ be the set of rays of a cone $\sigma$ and $n_{\rho}$ be the primitive lattice vector on the ray $\rho$. For $\rho\in\sigma(1)$ we set
$$
S_{\rho}:=\{e\in M\,|\, \langle n_{\rho},e\rangle=-1
\,\,\mbox{and}\,\, \langle n_{\rho'},e\rangle\geqslant0
\,\,\,\,\forall\,\rho'\in \sigma(1), \,\rho'\ne\rho\}.
$$
One easily checks that the set $S_{\rho}$ is infinite for each $\rho\in\sigma(1)$. The elements of the set $\mathfrak{R}:=\bigsqcup\limits_{\rho} S_{\rho}$ are called the {\itshape Demazure roots} of $\sigma$. Let $e\in S_{\rho}$. Then $\rho$ is called the {\itshape distinguished ray} of the root $e$. One can define the homogeneous LND on the algebra $A$ by the rule
$$
\partial_e(\chi^m)=\langle n_{\rho},m\rangle\chi^{m+e}.
$$
In fact, every homogeneous LND on~$A$ has a form $\alpha\partial_e$ for some $\alpha\in
\KK,\, e\in \mathfrak{R}$, see \cite[Theorem~2.7]{L1}. In other words, $\GG_a$-actions on $X$
normalized by the acting torus are in bijection with Demazure roots of the cone $\sigma$.

Clearly, all homogeneous LNDs on a toric variety are of fiber type.

\begin{example} \label{ex1}
Consider $X=\AA^k$ with the standard action of the torus $(\KK^{\times})^k$. It is a toric variety with the cone $\sigma=\QQ^k_{\geqslant0}$ having rays
$\rho_1=\langle(1,0,\ldots,0)\rangle_{\QQ_{\geqslant0}},\ldots,\rho_k=\langle(0,0,\ldots,0,1)\rangle_{\mathbb{Q}_{\geqslant0}}$.
The dual cone $\omega$ is $\QQ^k_{\geqslant0}$ as well. In this case
$$
S_{\rho_i}=\{(c_1,\ldots,c_{i-1},-1,c_{i+1},\ldots,c_k)\,|\,c_j\in\ZZ_{\geqslant0}\}.
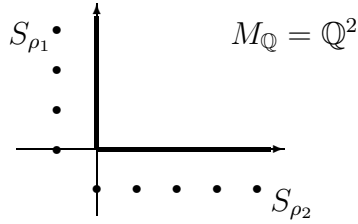
$$
\vspace{0.05cm}
\begin{center}
\begin{picture}(100,75)
\multiput(50,15)(15,0){5}{\circle*{3}}
\multiput(35,30)(0,15){4}{\circle*{3}}
\put(20,30){\vector(1,0){100}} \put(50,5){\vector(0,1){80}}
\put(17,70){$S_{\rho_1}$} \put(115,7){$S_{\rho_2}$}
\put(100,70){$M_{\mathbb{Q}}=\mathbb{Q}^2$} \linethickness{0.5mm}
\put(50,30){\line(1,0){65}} \put(50,30){\line(0,1){50}}
\end{picture}
\end{center}
Denote
$x_1=\chi^{(1,0,\ldots,0)},\ldots,x_k=\chi^{(0,\ldots,0,1)}$. Then
$\KK[X]=\KK[x_1,\ldots,x_k]$. It is easy to see that the homogeneous LND corresponding to the root
$e=(c_1,\ldots,c_k)\in S_{\rho_i}$ is
$$
\partial_e=x_1^{c_1}\ldots x_{i-1}^{c_{i-1}} x_{i+1}^{c_{i+1}}\ldots x_{k}^{c_{k}}\frac{\partial}{\partial x_i}.
$$

This LND gives rise to the $\GG_a$-action
$$
x_i\mapsto x_i+sx_1^{c_1}\ldots x_{i-1}^{c_{i-1}} x_{i+1}^{c_{i+1}}\ldots x_k^{c_k}, \quad
x_j\mapsto x_j, \quad j\ne i, \quad s\in\GG_a.
$$
\end{example}

Let us recall that the {\itshape complexity} of an action of a torus $T$ on an irreducible variety $X$ is the codimension of a general $T$-orbit on $X$, or, equivalently, the transcendence degree of the field of rational invariants $\KK(X)^T$ over $\KK$. In particular, actions of complexity zero are precisely actions with an open $T$-orbit.

Now we recall a description of normal affine $T$-varieties of complexity one in terms of proper polyhedral divisors. Let $N$ and $M$ be two mutually dual lattices with the pairing denoted by $\langle,\rangle$, $\sigma$ be a strongly convex cone in $N_{\mathbb{Q}}$, and $\omega\subseteq M_{\mathbb{Q}}$ be the dual cone. A~polyhedron $\Delta\subseteq N_{\mathbb{Q}}$, which can be decomposed as Minkowski sum of a bounded polyhedron and the cone $\sigma$, is called {\itshape$\sigma$-tailed}. Let $C$ be a smooth curve.
A {\itshape$\sigma$-polyhedral} divisor on $C$ is a formal sum
$$
\mathfrak{D}=\sum\limits_{z\in C} \Delta_{z}\cdot z,
$$
where $\Delta_{z}$ are the $\sigma$-tailed polyhedra and only finite number of them are not
equal to $\sigma$. The divisor $\mathfrak{D}$ is {\itshape trivial}, if $\Delta_{z}=\sigma$
for all $z\in C$.

The finite set ${\supp \mathfrak{D}:=\{z\in C\mid \Delta_{z}\ne\sigma\}}$ is called the {\itshape support} of $\mathfrak{D}$. For every ${m\in\omega_{M}}$ we can obtain the
$\mathbb{Q}$-divisor $\mathfrak{D}(m)=\sum\limits_{z\in C} h_z(m)\cdot z,$ where $h_z(m):=\min\limits_{p\in\Delta_z} \langle p,m\rangle$. So a $\sigma$-polyhedral divisor is just a piecewise-linear function from $\omega_M$ to the group of $\QQ$-divisors on~$C$. One can define the $M$-graded algebra
$$
{A[C,\mathfrak{D}]=\bigoplus_{m\in \omega_{M}} A_m\chi^m,}
\,\,\,\,\mbox{where}\,\,\,
A_m=H^0(C,\mathfrak{D}(m)):=\{f\in\mathbb{K}(X)\mid \dv
f+\mathfrak{D}(m)\geqslant0\},
$$
where the multiplication of homogeneous elements is given as in $\KK(X)$.

A $\sigma$-polyhedral divisor on smooth curve $C$ is called
{\itshape proper} if either $C$ is affine, or $C$ is projective and
the polyhedron $\deg \mathfrak{D}:=\sum\limits_{z\in C}
\Delta_{z}$ is a proper subset of $\sigma$.

The next theorem expresses the main results of \cite{AH}
specialized to the case of torus actions of complexity one.

\begin{theorem} \label{AH}

$(1)$  Let $C$ be a smooth curve and $\mathfrak{D}$ a proper
$\sigma$-polyhedral divisor on~$C$. Then the $M$-graded algebra
$A[C,\mathfrak{D}]$ is a normal finitely generated effectively
graded \linebreak $(\rk M+1)$-dimensional domain. Conversely, for each normal
finitely generated domain $A$ with a grading of complexity one there exist a
smooth curve $C$ and a proper $\sigma$-polyhedral divisor
$\mathfrak{D}$ on $C$ such that $A$ is isomorphic to
$A[C,\mathfrak{D}]$.

$(2)$ The $M$-graded domains $\spec A[C,\mathfrak{D}]$ and $\spec A[C,\mathfrak{D}']$ are isomorphic  if and only if for every $z\in C$ there exists a lattice vector $v_{z}\in N$ such that
$$
\mathfrak{D}=\mathfrak{D}'+\sum_{z} (v_{z}+\sigma)\cdot z,
$$
and for all $m\in\omega_{M}$ the divisor $\sum\limits_{z}\langle v_{z},m\rangle\cdot z$ is principal.
\end{theorem}

The following result is obtained in~\cite[Section 11]{AH}.

\begin{proposition} \label{toric}
Let $\mathfrak{D}$ be a proper $\sigma$-polyhedral divisor on a smooth curve $C$,
$X=\spec A[C,\mathfrak{D}]$, and $T\times X\to X$ be the corresponding torus action.
Then this action can be realized as a subtorus action an a toric variety
if and only if either $C=\mathbb{A}^1$ and $\mathfrak{D}$ can be chosen supported in at most one point, or $C=\mathbb{P}^1$ and $\mathfrak{D}$ can be chosen supported in at most two points.
\end{proposition}

Also we need a description of homogeneous LNDs of horizontal type for a $T$-variety $X$ of complexity one from~\cite{L1}. Below we follow the approach given in \cite{AL}. We have
$\KK[X]=A[C,\mathfrak{D}]$ for some $C$ and~$\mathfrak{D}$. It turns out that $C$ is isomorphic to $\AA^1$ or $\PP^1$ whenever there exists a homogeneous LND of horizontal type on $A[C,\mathfrak{D}]$, see \cite[Lemma~3.15]{L1}.

Let $C$ be $\AA^1$ or $\PP^1$, $\mathfrak{D}=\sum\limits_{z\in C}\Delta_z\cdot z$ a $\sigma$-polyhedral divisor on $C$, $z_0\in C$, $z_{\infty}\in C\backslash\{z_0\}$, and $v_z$ a vertex of $\Delta_z$ for every $z\in C$. Put $C'=C$ if $C=\AA^1$ and $C'=C\backslash\{z_{\infty}\}$ if $C=\PP^1$. A collection $\widetilde{\mathfrak{D}}=\{\mathfrak{D},z_0; v_z, \forall z\in C\}$ if $C=\AA^1$ and $\widetilde{\mathfrak{D}}=\{\mathfrak{D},z_0,z_{\infty}; v_z, \forall z\in C'\}$ if $C=\PP^1$ is called a {\itshape colored } $\sigma$-polyhedral divisor on $C$ if the following conditions hold:

\smallskip

$(*)$ $v_{\deg}:=\sum\limits_{z\in C'} v_z$ is a vertex of
$\deg \mathfrak{D}\hspace{-0.15cm}\mid_{C'}:=\sum\limits_{z\in C'}\Delta_z$;

\smallskip

$(**)$ $v_z\in N$ for all $z\in C'$, $z\ne z_0$.

\smallskip

Let $\widetilde{\mathfrak{D}}$ be a colored $\sigma$-polyhedral divisor on $C$ and $\delta\subseteq N_{\QQ}$ be the cone generated by ${\deg\mathfrak{D}\hspace{-0.15cm}\mid_{C'}-v_{\deg}}$. Denote by $\widetilde{\sigma}\subseteq (N\oplus\ZZ)_{\QQ}$ the cone generated by $(\delta,0)$ and $(v_{z_0},1)$ if $C=\AA^1$, and by $(\delta,0)$, $(v_{z_0},1)$ and
$(\Delta_{z_{\infty}}+v_{\deg}-v_{z_0}+\delta,-1)$ if $C=\PP^1$. By definition, put $d$ the minimal positive integer such that $d\cdot v_{z_0}\in N$. A pair $(\widetilde{\mathfrak{D}},e)$, where $e\in M$, is said to be {\itshape coherent } if

\smallskip

\begin{itemize}
\item[(i)] there exists $s\in\ZZ$ such that $\widetilde{e}=(e,s)\in M\oplus\ZZ$ is a Demazure root of the cone $\widetilde{\sigma}$ with distinguished ray
    $\widetilde{\rho}=(d\cdot v_{z_0},d)$;

\smallskip

\item[(ii)] $\langle v,e\rangle\geqslant 1+\langle v_z,e\rangle$ for all
$z\in C'\backslash \{z_0\}$ and all vertices $v\ne v_{z}$ of the polyhedron $\Delta_{z}$;

\smallskip

\item[(iii)] $d\cdot \langle v,e\rangle\geqslant 1+\langle v_{z_0},e\rangle$ for all vertices $v\ne v_{z_0}$ of the polyhedron $\Delta_{z_0}$;

\smallskip

\item[(iv)] if $Y=\mathbb{P}^1$, then
$d\cdot \langle v,e\rangle\geqslant -1-d\cdot \sum\limits_{z\in Y'} \langle v_{z},e\rangle$ for all vertices $v$ of the polyhedron~$\Delta_{z_{\infty}}$.
\end{itemize}

\smallskip

It follows from \cite[Theorem~1.10]{AL} that homogeneous LNDs of horizontal type on $A[C,\mathfrak{D}]$ are in bijection with the coherent pairs $(\widetilde{\mathfrak{D}},e)$. Namely, let $(\widetilde{\mathfrak{D}},e)$ be a coherent pair. Without loss of generality we may assume that $z_0=0$, $z_{\infty}=\infty$ if $C=\PP^1$, and $v_z=0\in N$ for all $z\in
C'\backslash\{z_0\}$. Let $\KK(C)=\KK(t)$. Then the homogeneous LND of horizontal type corresponding to $(\widetilde{\mathfrak{D}},e)$ is given by

\smallskip

\begin{equation} \label{lnd}
\partial(\chi^m\cdot t^r)=d(\langle v_0,m\rangle+r)\chi^{m+e}\cdot
t^{r+s}\,\,\,\,\,\,\,\, \mbox{for all}\,\,\, m\in M, r\in\ZZ.
\end{equation}

\smallskip

In particular, the vector $e$ is the degree of the derivation $\partial$.


\section{The affine case} \label{sec3}
Let $(X,x)$ be an equivariant embedding of the group $\GG=(\GG_m)^{n-1}\times\GG_a$, where
$n=\dim X$. In this section we assume that $X$ is normal and affine. Let us denote the subgroup
$(\GG_m)^{n-1}$ of $\GG$ by $T$. Since the action of $T$ on $X$ is effective, it has compelxity one and defines an effective grading of the algebra $\KK[X]$ by the lattice $M$. In particular, the graded algebra $\KK[X]$ has the form $A[C,\mathfrak{D}]$ for some smooth curve $C$ and some proper $\sigma$-polyhedral divisor on $C$, where $\sigma$ is a cone in $N_{\QQ}$.

Since the action of the subgroup $\GG_a$ commutes with $T$-action on $X$, the corresponding homogeneous LND on $\KK[X]$ has degree zero. Moreover, the group $\GG$ acts on $X$ with an open orbit. It implies that the $\GG_a$-action on $X$ is of horizontal type, and hence either $C=\mathbb{A}^1$ or $C=\mathbb{P}^1$.

\begin{proposition} \label{deg0}
Let $X=\spec A[C,\mathfrak{D}]$ be a $T$-variety of complexity one. Suppose that there exists a homogeneous LND of horizontal type and of degree zero on $A[C,\mathfrak{D}]$. Then
\begin{itemize}
\item[(1)]
if $\,\,C=\mathbb{A}^1$, then one can assume (via Theorem \ref{AH}) that $\mathfrak{D}$ is a trivial $\sigma$-polyhedral divisor;
\item[(2)]
if $\,\,C=\mathbb{P}^1$, then one can choose $\mathfrak{D}=\Delta_{\infty}\cdot[\infty]$,
where $\Delta_{\infty}\varsubsetneq \sigma$ is some $\sigma$-tailed polyhedron.
\end{itemize}
\end{proposition}

\begin{proof}
Let $(\widetilde{\mathfrak{D}},0)$ be the coherent pair corresponding to the homogeneous LND of horizontal type. Without loss of generality we may assume that $z_0=0$ and $z_{\infty}=\infty$ if $C=\mathbb{P}^1$. By definition of a coherent pair, there exists $s\in \mathbb{Z}$ such that $(0,s)$ is a Demazure root of the cone $\widetilde{\sigma}$ with distinguished ray $(dv_0,d)$. It implies that $s=-1$, $d=1$, and hence $v_0\in N$. Further, the inequality ${\langle v,0\rangle\geqslant 1+\langle v_z,0\rangle}$ should be satisfied for every $z\in C'$ and every vertex $v\ne v_{z}$ of $\Delta_{z}$. It means that each polyhedron $\Delta_{z}$, where $z\in C'$, has only one vertex $v_z$. Replacing $\sigma$-polyhedral divisor $\mathfrak{D}$ with ${\mathfrak{D}'=\mathfrak{D}+\sum_{z\in C'} (-v_{z}+\sigma)\cdot z}$ and using Theorem \ref{AH}, we obtain the assertion. The condition $\Delta_{\infty}\varsubsetneq  \sigma$ follows from the fact that $\mathfrak{D}$ is a proper $\sigma$-polyhedral divisor.
\end{proof}

\begin{corollary} \label{tor}
Under the conditions of Proposition~\ref{deg0} the variety $X$ is toric with $T$ being a subtorus of the action torus $\TT$.
\end{corollary}

\begin{proof}
It follows immediately from Propositions~\ref{toric} and~\ref{deg0}.
\end{proof}

The next proposition is a specification of Corollary~\ref{tor}. In particular, it shows that the $\mathbb{G}_a$-action on $X$ is normalized by the acting torus $\mathbb{T}$.

\begin{proposition} \label{propaff}
Under the conditions of Proposition~\ref{deg0},
\begin{itemize}
\item[(1)]
if $\,\,C=\AA^1$, then $X\cong Y\times\AA^1$, where $Y$ is the toric variety corresponding to the cone $\sigma$ and $\GG_a$ acts on $\AA^1$ by translations;
\item[(2)]
if $C=\PP^1$, then $X$ is the toric variety with the cone
$\widetilde{\sigma}\subset N\oplus\ZZ$ generated by $(\sigma,0),(\Delta_{\infty},-1)$ and $(0,1)$. The $\GG_a$-action on $X$ is given by Demazure root
$\widetilde{e}=(0,-1)\in M\oplus\ZZ$ of the cone $\widetilde{\sigma}$.
\end{itemize}
\end{proposition}

\begin{proof} Let $\mathbb{K}(C)=\mathbb{K}(t)$. If $C=\mathbb{A}^1$ then $\mathfrak{D}$ is trivial and
$$
A[C,\mathfrak{D}]=\bigoplus_{m\in \omega_M} \mathbb{K}[t]\cdot\chi^m=\mathbb{K}[\omega_M]\otimes\mathbb{K}[t]=
\mathbb{K}[Y]\otimes\mathbb{K}[t].
$$
Hence $X\cong Y\times\AA^1$. Applying formula \eqref{lnd}, we obtain that the homogeneous LND is given by
\begin{equation}\label{lnd'}
\partial(\chi^m\cdot t^r)=r\chi^{m}\cdot t^{r-1}
\end{equation}
for all $m\in \omega_M$ and $r\in\ZZ_{\geqslant0}$. Thus
$\GG_a$ acts on $Y\times\AA^1$ as $(y,t)\mapsto (y,t+s)$.

\smallskip

If $C=\PP^1$ then $\mathfrak{D}=\Delta_{\infty}\cdot[\infty]$ and we obtain
$$
A[C,\mathfrak{D}]=\bigoplus_{m\in \omega_M}\bigoplus_{r=0}^{h_{\infty}(m)}\KK\chi^m \cdot t^r=
\bigoplus_{(m,r)\in \widetilde{\omega}_{\widetilde{M}}}
\KK\chi^m\cdot t^r=\KK[\widetilde{\omega}_{\widetilde{M}}],
$$
where $\widetilde{M}=M\oplus\ZZ$ and $\widetilde{\omega}\subset\widetilde{M}_{\QQ}$ is the cone
dual to $\widetilde{\sigma}$. So we see that $A[C,\mathfrak{D}]$ is a semigroup algebra and $X$ is a toric variety with the cone $\widetilde{\sigma}$. In this case formula~\eqref{lnd'} gives the
LND corresponding to the Demazure root $\widetilde{e}=(0,-1)$.
\end{proof}


\section{The complete case} \label{sec4}

In this section we study equivariant compactifications of the group $\GG$.
First we briefly recall the main ingredients of the Cox construction, see~\cite[Chapter~1]{ADHL} for more details.

Let $X$ be a normal variety with finitely generated divisor class group $\cl(X)$ and only constant invertible regular functions.

Suppose that $\cl(X)$ is free. Denote by $\wdiv(X)$ the group of Weil divisors on $X$ and fix a subgroup $K \subseteq \wdiv(X)$ which maps onto $\cl(X)$ isomorphically. The {\itshape Cox ring} of the variety $X$ is defined as
$$
R(X)=\bigoplus_{D\in K} H^0(X,D),
$$
where $H^0(X,D)=\{f\in\KK(X)\mid \dv f+D\geqslant0\}$ and multiplication on homogeneous components coincides with multiplication in $\KK(X)$ and extends to $R(X)$ by linearity.

If $\cl(X)$ has torsion, we choose finitely generated subgroup $K\subseteq\wdiv(X)$ that projects to $\cl(X)$ surjectively. Denote by $K_0\subset K$ the kernel of this projection. Take compatible
bases $D_1,\ldots,D_s$ and $D_1^0=d_1D_1,\ldots,D_r^0=d_rD_r$ in $K$ and $K_0$ respectively. Let us choose the set of rational functions $\EuScript{F}=\{F_D\in\KK(X)^{\times} \,:\, D\in K_0\}$ such that $\dv(F_D)=D$ and $F_{D+D'}=F_DF_{D'}$. Suppose that $D,D'\in K$ and $D-D'\in K_0$. A map $f\mapsto F_{D-D'}f$ is an isomorphism of the vector spaces $H^0(X,D)$ and $H^0(X,D')$. The linear span of the elements $f-F_{D-D'}f$ over all $D,D'$ with $D-D'\in K_0$ and all $f\in H^0(X,D)$ is an ideal $I(K,\EuScript{F})$ of the graded ring $T_K(X):=\bigoplus_{D\in K}
H^0(X,D)$. The Cox ring of the variety $X$ is given by
$$
R(X)=T_K(X)/I(K,\EuScript{F}).
$$
This construction does not depend on the choice of $K$ and $\EuScript{F}$,
see~\cite[Lemma~3.1 and Proposition~3.2]{Arz}.

Suppose that the Cox ring $R(X)$ is finitely generated. Then ${\overline{X}:=\spec R(X)}$ is a normal affine variety with an action of the  quasitorus $H_{X} := \spec\KK[\cl(X)]$. There is an open $H_{X}$-invariant subset $\widehat{X}\subseteq \overline{X}$ such that the complement $\overline{X}\backslash\widehat{X}$ is of codimension at least two in $\overline{X}$,
there exists a good quotient $p_X\colon\widehat{X}\rightarrow\widehat{X}/\!/H_{X}$, and the quotient space $\widehat{X}/\!/H_{X}$ is isomorphic to $X$. So we have the following diagram
$$
\begin{CD}
\widehat{X} @>{i}>> \overline{X}=\spec R(X)\\
@VV{/\!/H_{X}}V  \\
X
\end{CD}
$$
Let us return to equivariant compactifications of $\GG$.

\begin{proposition} \label{pcomplete}
Let $\GG=T\times\GG_a$ and $X$ be a normal compactification of $\GG$. Then the $T$-action on $X$ can be extended to an action of a bigger torus $\TT$ such that $\TT$ normalizes
$\GG_a$ and $X$ is a toric variety with the acting torus $\TT$.
\end{proposition}

\begin{proof} The variety $X$ is rational with torus action of complexity one. By~\cite[Theorem~4.3.1.5]{ADHL}, the divisor class group $\cl(X)$ and the Cox ring $R(X)$ are finitely generated.

There exists a finite epimorphism $\epsilon\colon\GG'\to\GG$ of connected linear algebraic groups and an action $\GG'\times\widehat{X}\to\widehat{X}$ which commutes with the quasitorus $H_X$
and $p_X(g'\cdot\widetilde{x})=\epsilon(g')\cdot p_X(\widetilde{x})$ for all $g'\in\GG'$ and $\widetilde{x}\in\widetilde{X}$, see~\cite[Theorem~4.2.3.1]{ADHL}. The group $\GG'$ has a form
$T'\times\GG_a$, where $\epsilon$ defines a finite epimorphism of tori $T'\to T$ and is identical
on $\GG_a$.

Since $\overline{X}=\spec\KK[\widehat{X}]$, the action of $\GG'$ extends to the affine variety $\overline{X}$. This variety is an embedding of the group $(T'H_X^0)\times\GG_a$. By Proposition~\ref{propaff}, it is toric with an acting torus $\overline{\TT}$ normalizing the $\GG_a$-action and $T'H_{X}^0$ is a subtorus of $\overline{\TT}$. Since $X$ is complete, \cite[Corollary~2.5]{Sw} implies that the subset $\widetilde{X}$ is invariant under the torus $\overline{\TT}$. By~\cite[Lemma~4.2.1.3]{ADHL}, the action of $\overline{\TT}$ descends to
an action of the torus $\TT:=\overline{\TT}/H_X^0$ on $X$. Here $\TT$ normalizes $\GG_a$, its action extends the action of $T$ on $X$, and $X$ is toric with respect to $\TT$.
\end{proof}


\section{Toric varieties and Demazure roots} \label{sec5}

We keep notations of Section~\ref{sec2}. Let $X$ be a toric variety of dimension $n$ with an acting torus $\TT$ and $\Sigma$ be the corresponding fan of convex polyhedral cones in the space $N_{\QQ}$, see~\cite{Fu} or \cite{CLS} for details.

As before, let $\Sigma(1)$ be the set of rays of the fan $\Sigma$ and $n_{\rho}$ be the primitive lattice vector on the ray $\rho$. For $\rho\in\Sigma(1)$ we consider the set $S_{\rho}$ of all vectors $e\in M$ such that
\begin{enumerate}
\item
$\langle n_{\rho},e\rangle=-1\,\,\mbox{and}\,\, \langle n_{\rho'},e\rangle\geqslant0
\,\,\,\,\forall\,\rho'\in \sigma(1), \,\rho'\ne\rho$;
\smallskip
\item
if $\sigma$ is a cone of $\Sigma$ and $\langle v,e\rangle=0$ for all $v\in\sigma$, then the cone generated by $\sigma$ and $\rho$ is in $\Sigma$ as well.
\end{enumerate}

Note that condition~(1) implies condition~(2) if $\Sigma$ is a maximal fan with support $|\Sigma|$. This is the case if $X$ is affine or complete.

The elements of the set $\mathfrak{R}:=\bigsqcup\limits_{\rho} S_{\rho}$ are called the {\itshape Demazure roots} of the fan $\Sigma$, cf.~\cite[Definition~4]{De} and \cite[Section~3.4]{Oda}.
Again elements $e\in\mathfrak{R}$ are in bijections with $\GG_a$-actions on $X$ normalized by the acting torus. Let us denote the corresponding one-parameter subgroup of $\aut(X)$ by $H_e$.

We recall basic facts from toric geometry. There is a bijection between cones $\sigma\in\Sigma$ and $\TT$-orbits $\OOO_{\sigma}$ on $X$ such that $\sigma_1\subseteq\sigma_2$ if and only if $\OOO_{\sigma_2}\subseteq\overline{\OOO_{\sigma_1}}$. Here $\dim\OOO_{\sigma}=n-\dim\langle\sigma\rangle$. Moreover, each cone $\sigma\in\Sigma$ defines an open affine $\TT$-invariant subset $U_{\sigma}$ on $X$ such that $\OOO_{\sigma}$ is a unique closed $\TT$-orbit on $U_{\sigma}$ and $\sigma_1\subseteq\sigma_2$ if and only if $U_{\sigma_1}\subseteq U_{\sigma_2}$.

Let $\rho_e$ be the distinguished ray corresponding to a root $e$, $n_e$ be the primitive lattice vector on $\rho_e$, and $R_e$ be the one-parameter subgroup of $\TT$ corresponding to $n_e$.

\smallskip

Our aim is to describe the action of $H_e$ on $X$.

\begin{proposition} \label{phe}
For every point $x\in X\setminus X^{H_e}$ the orbit $H_ex$ meets exactly two $\TT$-orbits $\OOO_1$ and $\OOO_2$ on $X$, where $\dim\OOO_1=\dim\OOO_2+1$. The intersection $\OOO_2\cap H_ex$ consists of a single point, while
$$
\OOO_1\cap H_ex=R_ey \quad \text{for any} \quad y\in\OOO_1\cap H_ex.
$$
\end{proposition}

\begin{proof}
It follows from the proof of \cite[Proposition~3.14]{Oda} that the affine charts $U_{\sigma}$, where $\sigma\in\Sigma$ is a cone containing $\rho_e$, are $H_e$-invariant, and the complement of their union is contained in $X^{H_e}$, cf. \cite[Lemma~2.4]{Ba}. This reduces the proof to the case $X$ is affine. Then the assertion is proved in \cite[Proposition~2.1]{AKZ}.
\end{proof}

A pair of $\TT$-orbits $(\OOO_1,\OOO_2)$ on $X$ is said to be {\itshape $H_e$-connected} if
$H_ex\subseteq \OOO_1\cup\OOO_2$ for some $x\in X\setminus X^{H_e}$. By Proposition~\ref{phe}, $\OOO_2\subseteq\overline{\OOO_1}$ for such a pair (up to permutation) and $\dim\OOO_2=\dim\OOO_1+1$. Since the torus normalizes the subgroup $H_e$, any point of
$\OOO_1\cup\OOO_2$ can actually serve as a point $x$.

\begin{lemma} \label{lpe}
A pair of $\TT$-orbits $(\OOO_{\sigma_1},\OOO_{\sigma_2})$ is $H_e$-connected if and only if
$e|_{\sigma_2}\le 0$ and $\sigma_1$ is a facet of $\sigma_2$ given by the equation
$\langle v,e\rangle=0$.
\end{lemma}

\begin{proof} The proof again reduces to the affine case, where the assertion is \cite[Lemma~2.2]{AKZ}.
\end{proof}


\section{The orbit structure} \label{sec6}

We keep notations of the previous section.
Let us begin with a construction mentioned in the Introduction. Let $X$ be a toric variety with the acting torus $\TT$. Consider a non-trivial action $\GG_a\times X\to X$ normalized by $\TT$ and thus represented by a Demazure root $e$ of the fan $\Sigma$ of $X$. Then $\TT$ acts on $\GG_a$ by conjugation with the character $e$ and the semidirect product $\TT\rightthreetimes\GG_a$ acts on $X$ as well. Let $T=\Ker(e)\subseteq\TT$ and consider the group $\GG:=T\times\GG_a$.

\begin{proposition} \label{constr}
The variety $X$ is an embedding of $\GG$.
\end{proposition}

\begin{proof}
Take a point $x\in X$ whose stabilizers in $\TT$ and $\GG_a$ are trivial. It suffices to show that the stabilizer of $x$ in $\GG$ is trivial. To this end, note that by the Jordan decomposition~\cite[Theorem~15.3]{Hum} any subgroup of $T\times\GG_a$ is a product of subgroups in $T$ and $\GG_a$ respectively.
\end{proof}

\begin{remark}
The $\GG$-embedding of Proposition~\ref{constr} is defined by the pair $(\Sigma, e)$.
\end{remark}

Since $\langle n_e,e\rangle=-1$, we have $\TT=T\times R_e$.

\begin{lemma} \label{inv}
Any $(T\times\GG_a)$-invariant subset in $X$ is also $\TT$-invariant.
\end{lemma}

\begin{proof}
Note that an orbit $\TT x$ does not coincide with the orbit $Tx$ if and only if the stabilizer of $x$ in  $\TT$ is contained in $T$. For $x\in\OOO_{\sigma}$ this condition is equivalent to $e|_{\sigma}=0$. It shows that for every $x\in X^{\GG_a}$ we have $\TT x=Tx$. If
$x\in X\setminus X^{\GG_a}$, then by Proposition~\ref{phe} the orbit $\GG_ax$ is invariant under
$R_e$. This proves that any orbit of $(T\times\GG_a)$ is $R_e$- and $\TT$-invariant, thus the assertion.
\end{proof}

\begin{proposition} \label{pclo}
Let $X$ be a $\GG$-embedding given by a pair $(\Sigma,e)$. Then any $\GG$-orbit on $X$ is either a union $\OOO_1\cup\OOO_2$ of two $\TT$-orbits on $X$ or a unique $\TT$-orbit; the first possibility occurs if and only if the pair $(\OOO_1,\OOO_2)$ is $H_e$-connected. In particular, the  number of $\GG$-orbits on $X$ is finite.
\end{proposition}

\begin{proof}
The assertion follows directly from Lemma~\ref{inv} and Proposition~\ref{phe}.
\end{proof}

\begin{proposition}
Let $X$ be a $\GG$-embedding given by a pair $(\Sigma,e)$. Then the stabilizer of any point
$x\in X$ in $\GG$ is connected and the closure of any $\GG$-orbit on $X$ is a (normal) toric variety. If $X$ is smooth, then the closure of any $\GG$-orbit is smooth.
\end{proposition}

\begin{proof}
The stabilizer of a point $x$ in $\GG$ is the direct product of stabilizers in $T$ and
in~$\GG_a$. An algebraic subgroup of $\GG_a$ is either $\{0\}$ or $\GG_a$ itself, while the stabilizer in $T$ is the kernel of the (primitive) character $e$ restricted to the (connected) stabilizer of $x$ in $\TT$. Thus the stabilizer of $x$ in $\GG$ is connected.

Proposition~\ref{pclo} shows that any $\GG$-orbit on $X$ contains an open $\TT$-orbit, and thus the closure of a $\GG$-orbit coincides with the closure of some $\TT$-orbit. Now the last two assertions follow from \cite[Section~3.1]{Fu}.
\end{proof}

\begin{remark}
If $X$ contains $l$ torus invariant prime divisors, then the number of $\GG$-invariant prime divisors on $X$ is $l-1$. On a toric variety, the closure of any torus orbit is an intersection of torus invariant prime divisors. In contrast, not every $\GG$-orbit closure on $X$ is an intersection of $\GG$-invariant prime divisors, see Example~\ref{ex3}.
\end{remark}


\section{The general case} \label{sec7}

We are going to show that every $\GG$-embedding can be realized as in Proposition~\ref{constr}.

\begin{theorem} \label{tmain}
Let $\GG=T\times\GG_a$ and $X$ be a normal equivariant $\GG$-embedding. Then the $T$-action on $X$ can be extended to an action of a bigger torus $\TT$ such that $\TT$ normalizes
$\GG_a$ and $X$ is a toric variety with the acting torus $\TT$. In particular, every $\GG$-embedding comes from a pair $(\Sigma,e)$, where $\Sigma$ is a fan and $e$ is a Demazure root of $\Sigma$.
\end{theorem}

\begin{proof}
We begin with a classical result of Sumihiro. Let $X$ be a normal variety with a regular action $G\times X\to X$ of a linear algebraic group $G$. By~\cite[Theorem~3]{Su}, there exists a normal complete $G$-variety $\bf{X}$ such that $X$ can be embedded equivariantly as an open subset of $\bf{X}$. In other words, $\bf{X}$ is an equivariant compactification of $X$.

Let $X$ be a normal embedding of $\GG$ and $\bf{X}$ be an equivariant compactification of $X$. By Proposition~\ref{pcomplete}, the $T$-action on $\bf{X}$ can be extended to an action of a bigger torus $\TT$ such that $\TT$ normalizes $\GG_a$ and $\bf{X}$ is a toric variety with the acting torus $\TT$. Since the subset $X\subseteq\bf{X}$ is $(T\times\GG_a)$-invariant, it is invariant under $\TT$, see Lemma~\ref{inv}. This provides the desired structure of a toric variety on $X$.
\end{proof}

\begin{proposition}
A complete toric variety $X$ admits a structure of a $\GG$-embedding if and only if
$\aut(X)^0\ne\TT$.
\end{proposition}

\begin{proof}
The variety $X$ admits a structure of a $\GG$-embedding if and only if $\aut(X)^0$ contains at least one root subgroup. It is well known that the group $\aut(X)^0$ is generated by $\TT$ and root subgroups \cite[Proposition~11]{De}, \cite[Section~3.4]{Oda}, \cite[Corollary~4.7]{Cox}.
\end{proof}

Consider two structures of a $\GG$-embedding on a variety $X$. We say that such structures are {\itshape equivalent}, if there is an automorphism of $X$ sending one structure to the other.
Since the structure of a toric variety on $X$ is unique up to automorphism, we may assume that our two structures share the same acting torus $\TT$ and the same fan $\Sigma$, and are given by two roots $e,e'$ of $\Sigma$. Then the structures are equivalent if and only if $e$ can be sent to $e'$ by an automorphism of the torus $\TT$. This leads to the following result.

\begin{proposition} \label{peq}
Two structures of a $\GG$-embedding given by pairs $(\Sigma, e)$ and $(\Sigma, e')$ are equivalent if and only if there is an automorphism $\phi$ of the lattice $N$ which preserves the fan $\Sigma$ and such that the induced automorphism $\phi^*$ of the dual lattice $M$ sends $e$ to $e'$.
\end{proposition}

Let us finish with explicit examples of $\GG$-embeddings into a given variety.

\begin{example} \label{ex2}
We find all structures of $\GG$-embeddings on $\AA^2$. The cone of $\mathbb{A}^2$ as a toric variety is $\QQ^2_{\geqslant0}$. The set of Demazure roots of $\QQ^2_{\geqslant0}$ is
$$
\mathfrak{R}=\{(-1,k)\mid k\in\ZZ_{\geqslant0}\}\sqcup\{(k,-1)\mid k\in\ZZ_{\geqslant0}\},
$$
see Example~\ref{ex1}. The $\GG$-action on $\AA^2$ corresponding to the root $(-1,k)$ is given by
\begin{equation}\label{a2}
(t,s)\circ(x_1,x_2)=(t^kx_1+st^kx_2^k,tx_2),
\end{equation}
where $(x_1,x_2)\in\AA^2$, $s\in\GG_a$, and $t\in\KK^{\times}$. If $k\ne 0$, then there is a line of $\GG_a$-fixed points and the stabilizer of a non-zero point on this line is a cyclic group of order $k$. If $k=0$, then there is no $\GG_a$-fixed point. So formula~\eqref{a2} gives
non-equivalent $\GG$-actions for different~$k$. With $k\ne 0$ we have three $\GG$-orbits on $\AA^2$, while for $k=0$ there are two $\GG$-orbits.

Note that $\GG$-actions defined by the roots $(k,-1)$ and $(-1,k)$ are equivalent via the automorphism $x_1\leftrightarrow x_2$ of $\AA^2$.
\end{example}

\begin{example} \label{ex3}
Let $X=\PP^2$. It is a complete toric variety with a fan $\Sigma$ generated by the vectors $(1,0)$, $(0,1)$ and $(-1,-1)$:

\medskip

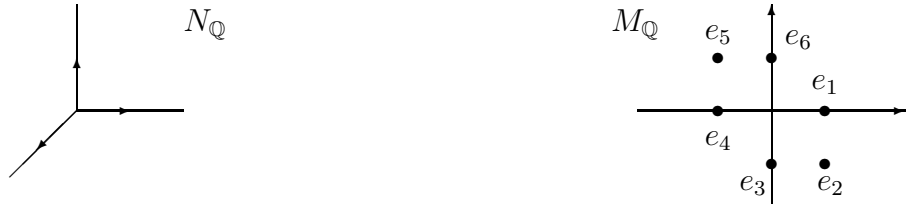
\begin{picture}(100,80)
\put(90,35){\vector(1,0){20}} \put(90,35){\vector(0,1){20}}
\put(90,35){\vector(-1,-1){15}} \put(130,65){$N_\QQ$}
\put(290,65){$M_{\QQ}$}
 \put(90,35){\line(1,0){40}}
\put(90,35){\line(0,1){40}} \put(90,35){\line(-1,-1){25}}
\put(325,22){$e_4$} \put(330,35){\circle*{4}} \put(365,43){$e_1$}
\put(370,35){\circle*{4}} \put(325,62){$e_5$}
\put(330,55){\circle*{4}} \put(367,5){$e_2$}
\put(370,15){\circle*{4}} \put(355,60){$e_6$}
\put(350,55){\circle*{4}} \put(338,5){$e_3$}
\put(350,15){\circle*{4}} \put(300,35){\vector(1,0){100}}
\put(350,0){\vector(0,1){75}}
\end{picture}

The set of Demazure roots is
$$
\mathfrak{R}=\{e_1=(1,0), e_2=(1,-1), e_3=(0,-1), e_4=(-1,0), e_5=(-1,1), e_6=(0,1)\}.
$$
We see that for any $i$ and $j$ there exists isomorphism of the fan $\Sigma$
sending $e_i$ to $e_j$. So any $\GG$-embedding into $\PP^2$ is equivalent to
$$
(t,s)\circ[z_0:z_1:z_2]=[tz_0+stz_1:tz_1:z_2].
$$
This time seven $\TT$-orbits glue to five $\GG$-orbits.
\end{example}

\begin{example} \label{ex4}
Consider the Hirzebruch surface $\FF_1$. The corresponding complete fan $\Sigma$ is generated by the vectors $(1,0)$, $(0,1)$, $(0,-1)$, and $(-1,1)$:

\medskip

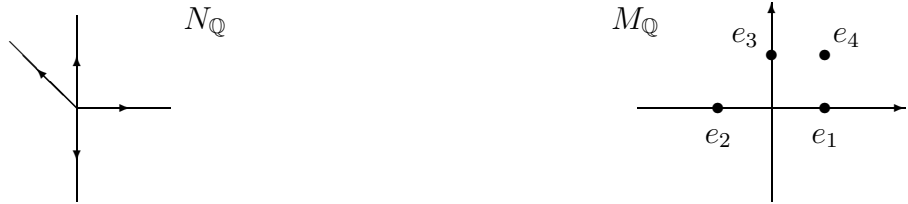
\begin{picture}(100,75)
\put(130,65){$N_\QQ$} \put(290,65){$M_{\QQ}$}
\put(90,35){\vector(1,0){20}} \put(90,35){\vector(0,1){20}}
\put(90,35){\vector(-1,1){15}} \put(90,35){\vector(0,-1){20}}
\put(90,35){\line(1,0){35}} \put(90,35){\line(0,1){35}}
\put(90,35){\line(-1,1){25}} \put(90,35){\line(0,-1){35}}
\put(330,35){\circle*{4}} \put(370,35){\circle*{4}}
\put(370,55){\circle*{4}} \put(350,55){\circle*{4}}
\put(300,35){\vector(1,0){100}} \put(350,0){\vector(0,1){75}}
\put(325,22){$e_2$} \put(365,22){$e_1$} \put(335,60){$e_3$}
\put(373,60){$e_4$}
\end{picture}

The set of Demazure roots is
$$
\mathfrak{R}=\{e_1=(1,0), e_2=(-1,0), e_3=(0,1), e_4=(1,1) \}.
$$
By an automorphism, we can send $e_1$ to $e_2$ and $e_3$ to $e_4$. For the first equivalence class we have six $\GG$-orbits, while in the second one the number of $\GG$-orbits in seven.
\end{example}


\end{document}